\documentclass{birkmult}
\usepackage[english]{babel}
\usepackage{amsfonts,amsmath,amsxtra,amsthm,amssymb,latexsym,bbm}
\usepackage[cp866]{inputenc}

\def\theequation {\thesection.\arabic{equation}}
\makeatletter\@addtoreset {equation}{section}\makeatother
\theoremstyle{plain}

\newtheorem{theorem}{Theorem}[section]

\newtheorem{lemma}{Lemma}[section]
\newtheorem{proposition}{Proposition}[section]

\newtheorem{remark}{Remark}[section]

\renewcommand{\proofname}{Proof}

\def\Re{\operatorname{Re}}

\begin{document}

\def\theequation {\thesection.\arabic{equation}}

\title{On existence of maximal semidefinite invariant subspaces for  $J$-dissipative operators}
\author[Ugra State University (Hanty-Mansiisk), Sobolev Institute of
Mathematics (Novosibirsk)]{S.G. Pyatkov}

\address{%
Ugra State University,\\  Chekhov st. 16, 628012,\\
Hanty-Mansiisk, Russia,\\fax:+7 3467357734,\\
phone: +7 9129010471}

\email{pyatkov@math.nsc.ru, s\_pyatkov@ugrasu.ru}



\date{July 23, 2010}

\begin{abstract} {\bf Abstract.}
We describe necessary and sufficient conditions for a
$J$-dissipative operator in a Krein space to have maximal
semidefinite invariant subspaces. The semigroup properties of
the restrictions  of an operator to these subspaces are studied. Applications are given to the case
when an operator admits matrix representation with respect to the
canonical decomposition of the space.
The main conditions are given in the terms of the interpolation theory of Banach spaces.
\end{abstract}

 \maketitle

\markright{Maximal semidefinite invariant suspaces}

\section{Introduction}

In this article we consider the question of existence of invariant
semidefinite invariant subspaces for  $J$-dissipative operators defined in a Krein space.
Recall that a Krein space  (see \cite{azi}) is a Hilbert space  $H$, where
together with the usual inner product  $(\cdot,\cdot)$ an indefinite inner product  (an indefinite metric)
 $[x, y] = (Jx, y)$, with $J=P^+-P^-$ ($P^{\pm}$ are orthoprojections in  $H$, $P^+ + P^-=I$) is introduced (see \cite{azi}).
We put  $H^{\pm}=R(P^{\pm})$. Here and in what follows, the symbol $I$
stands for the identity. The operator  $J$ is called a fundamental symmetry.
The Krein space is called a Pontryagin space if
 ${\rm dim\,} R(P^+)<\infty$ or  ${\rm dim\,} R(P^-)<\infty$ and it is denoted by $\Pi_{\kappa}$, where
 $\kappa=\min({\rm dim\,} R(P^+),{\rm dim\,} R(P^+))$.
A subspace $M$ in $H$  is said to be nonnegative (positive, uniformly positive)
if the inequality $[x, x]\geq 0$ ($[x, x]>0$, $[x, x]\geq \delta\|x\|^2$ $(\delta>0)$) holds
for all $x \in M$. Nonpositive, negative, uniformly negative subspaces in $H$
 are defined in a similar way.
If a nonnegative subspace $M$ admits no nontrivial nonnegative extensions, then it
is called a maximal nonnegative subspace. Maximal nonpositive (positive, negative, nonnegative, etc.) subspaces in $H$
 are defined by analogy.
 A densely defined operator $A$
is said to be dissipative (strictly dissipative, uniformly dissipative) in $H$
 if $-{\rm Re\,}(Ax, x)\geq 0$ for all $x \in D(A)$ ($-{\rm Re\,}(Ax, x)> 0$ for all $x \in D(A)$ or
  $-{\rm Re\,}(Ax, x)\geq \delta\|u\|^2$ ($\delta>0$) for all $x \in D(A)$). Similarly,
a densely defined operator $A$ is called a $J$-dissipative (strictly $J$-dissipative or uniformly $J$-dissipative)
whenever the operator $JA$ is dissipative (strictly dissipative or uniformly dissipative).
A dissipative ($J$-dissipative) operator is said to be maximal dissipative (maximal $J$-dissipative) if it admits no nontrivial
dissipative ($J$-dissipative) extensions. Let $A:H\to H$ be a $J$-dissipative operator.
We say that a subspace $M\subset H$ is invariant under $A$
if $D(A)\cap M$ is dense in $M$ and $Ax\in M$ for all $x\in D(A)\cap M$.

The main question under consideration here is the question  on existence of
semidefinite (i.~e. of a definite sign) invariant subspaces for a given $J$-dissipative operator in a Krein space.

The first results in this direction were obtained in the Pontryagin article  \cite{pont}, where is was proven that
every $J$-selfadjoint operator in a Pontryagin space (let ${\rm dim\,}H^+=\kappa<\infty$))
has maximal nonnegative invariant subspace  $M$ (${\rm
dim\,M}=\kappa$) such that the spectrum of the restriction $L|_{M}$ lies in the closed upper half-plane.

After this paper, the problem on  existence of invariant maximal semidefinite subspaces
turned out to be a focus of an attention in the theory of operators in Pontryagin and Krein spaces.
For example, we can note the articles by Krein M.G. (1950, 1964), Langer H. (1960, 1961, 1971, 1975),
Azizov T.Ya (1971, 1973, 2009), Mennicken R., Shkalikov A.A. (1996, 1999, 2005, 2007), and some others.

The Pontryagin results are generalized for different classes of operators in
 \cite{kre1}-\cite{azi3}. A sufficiently complete bibliography and some results are presented in
 \cite{azi}. Among the recent articles we note the articles
 \cite{azi3}-\cite{shk3}, where the most general results were obtained.
In these articles the whole space $H$ is identified with
the Cartesian product $H^+\times H^-$ ($H^{\pm}=R(P^{\pm})$)
and an operator
$A$  with the matrix operator $A:H^+\times H^-\to H^+\times H^-$ of the form
\begin{equation}\label{eq01}
\begin{array}{c} A=\left(\begin{array}{cc}A_{11} & A_{12}\\ A_{21} & A_{22}\end{array}\right),\
\\ A_{11}=P^+LP^+,\ A_{12}=P^+LP^-,\
A_{21}=P^-LP^+,\ A_{22}=P^-LP^-.\end{array}
\end{equation}
In this case the fundamental symmetry is as follows
$J=\left[\begin{array}{cc}I& 0\\ 0 & -I\end{array}\right]$.
The basic condition of existence of a maximal nonnegative invariant subspace for an operator $L$ in \cite{shk3} is the condition
of compactness of the operator  $A_{12}(A_{22}-\mu)^{-1}$ for some  $\mu$ from the left half-plane.

 For operators of a general form we present necessary and sufficient conditions of existence
of maximal  semidefinite invariant subspaces.
In contrast to the abo\-ve-men\-tio\-ned articles, the main conditions are stated in terms of the interpolation theory of Banach spaces.
Next, we apply the results obtained to the study of the operators represented in  the form
(\ref{eq01}) and weaken in some simple cases  the condition of compactness of the operator  $A_{12}(A_{22}-\mu)^{-1}$ replacing it
with the conditions that  the operators $A_{12}$ and  $A_{21}$ are
subordinate in some sense to the operators  $A_{11}$ and $A_{22}$.
In  the sufficiency part, the results are a generalization of those of the article
\cite{pyat1} (see also \cite{pyat2,pyat3}).

\section{ Preliminaries.}\label{s1}

Given Banach spaces  $X,Y$, the symbol  $L(X,Y)$ denotes the space of linear continuous operators
defined on  $X$ with values in $Y$. If $X=Y$ then we write  $L(X)$ rather than  $L(X,X)$.
We denote by $\sigma(A)$ and $\rho(A)$ the spectrum
and the resolvent set of $A$, respectively. The symbols $D(A)$ and $R(A)$ designate the domain and the range of an operator $A$.
If $M\subset X$ is a subspace then by the restriction of  $L$ to $M$ we mean
the operator  $L|_{M}:M\to X$ with the domain  $D(L|_{M})=D(L)\cap M$ coinciding with  $L$ on $D(L|_{M})$.
 An operator $A$ such that $-A$ is dissipative  (maximal dissipative) is called accretive (maximal accretive). Hence,
taking the sign into account we can say that all statements valid for an accretive operator are true for
a dissipative operator as well. In what follows, we replace the word  "maximal" with the letter $m$ and thus
we write $m$-dissipative rather than maximal dissipative.
If $A$ is an operator in a Krein space  $H$ then we denote by  $A^*$ and $A^c$ the adjoint operators
with respect to an inner product and an indefinite inner product in $H$, respectively.
The latter possesses the usual properties of an adjoint operator (see \cite{azi}).
Let  $A_0$ and $A_1$ be two Banach spaces continuously embedded into
a topological linear space $E$: $A_{0}\subset E$, \
$A_{1}\subset E$. Such a pair $\{A_{0},A_{1}\}$ is called an interpolation couple. Recall the definition
of the interpolation space  $(A_0,A_1)_{\theta,q}$. We describe the $K$-method. For every  $t, 0<t<\infty$, the functional
$$
K(t,a,A_0,A_1)=K(t,a)=\inf_{a=a_0+a_1}(\|a_0\|_{A_0}+t\|a_1\|_{A_1}),
\ a_0\in A_0,\  a_1\in A_1,
$$
defines a norm in the space  $A_0+A_1$ equivalent to its conventional norm.
Let  $0<\theta<1$. For $1\leq q<\infty$,
$$
(A_0,A_1)_{\theta,q}=\{a|a\in A_0+A_1,
\|a\|_{(A_0,A_1)_{\theta,q}}=\Big(\int_{0}^{\infty}[t^{-\theta}K(t,a)]^{q}
\frac{dt}{t}\Big)^{1/q}<\infty\}.
$$

We now present some facts used below.

\begin{proposition}\label{pro1}
Let $H$ be a Hilbert space {\rm (}a Krein space{\rm )}.

{\rm 1.}   A maximal dissipative  {\rm (}$J$-dissipative{\rm )} operator  $A$ is always closed
and if  $A$ is  $m$-dissipative then ${\mathbb C}^+=\{z\in {\mathbb C}: \ \Re z>0\}\subset \rho(A)$
{\rm (}see \cite[Propos. C.7.2]{haase}{\rm )}.

{\rm 2.} If $A$ is  $m$-dissipative then the operators   $A+i\omega I$ {\rm (}$\omega\in {\mathbb R}${\rm )} and
$A-\varepsilon I$ $(\varepsilon>0)$ are $m$-dissipative  as well {\rm (\cite[Propos. C.7.2]{haase})}.

{\rm 3.} If $A$ is $m$-$J$-dissipative then so are the operators
$A+i\omega I$ {\rm (}$\omega\in {\mathbb R}${\rm )}.
{\rm (}This property easily follows from the definition{\rm )}.

{\rm 4.} If $A$ is  $m$-$J$-dissipative and  ${\rm ker\, }A=\{0\}$ then  ${\rm ker\, }A^c=\{0\}$
{\rm (}see \cite[Chap. 2, Sect. 2, Corollary 2.17]{azi},  note that in  \cite{azi}
the authors use a different definition of a $J$-dissipative  operator. In this definition it is required
that  ${\rm Im\,}[Au,u]\geq 0$ for all $u\in D(A)${\rm )}.

{\rm 5.}  If  $A$ is a maximal uniformly dissipative  {\rm (}$J$-disipative{\rm )} operator then
 $i{\mathbb R}\in \rho(A)$ {\rm (}see \cite[Chap. 2, Sect.2, Propos.  2.32]{azi}{\rm )}.

{\rm 6.} If  $A$ is  $m$-dissipative   {\rm (}$m$-$J$-dissipative {\rm )} then the operator  $A$ is injective
if and only if the subspace $R(A)$ is dense in  $H$ {\rm (}see \cite[Propos. 7.0.1]{haase}{\rm )}.

{\rm 7.} If  $A$ is $m$-dissipative   {\rm (}$m$-$J$-dissipative{\rm )} then so is  the operator  $A^*$ $(A^c)$
 {\rm (}\cite[Propos.  C.7.2]{haase}, or   \cite[Chap. 2, Sect. 2, Propos. 2.7]{azi}{\rm )}.
\end{proposition}

Let $H$ be a complex Hilbert space with the norm $\|\cdot\|$ and let $L:H\to H$ be a closed densely defined
operator such that $\rho(L)\neq \emptyset$. Take $\lambda\in \rho(L)$ and endow the space  $H_k=D((L-\lambda I)^k)$ ($k$
is a nonnegative integer) with the norm $\|u\|_{H_k}=\|(L-\lambda I)^ku\|$. If $k<0$
is an integer then we denote by  $H_{k}$  the completion of  $H$ with respect to the norm
 $\|u\|_{H_{k}}=\|(L-\lambda I)^ku\|$.
Using the Hilbert resolvent  identity
\begin{equation}\label{eq11}
(L-\lambda I)^{-1}-(L-\gamma I)^{-1}=(\lambda-\gamma)(L-\lambda I)^{-1}(L-\gamma I)^{-1},
\end{equation}
we can  prove that the norm in  $H_{k}$ is independent of the parameter $\lambda\in \rho(L)$ and thus the definition
 of this space is correct.
Probably for the first time,
these spaces were introduced and studied in the articles by Grisvard (1966).
 At present the spaces $H_k$ for $k<0$ are often called extrapolation spaces and
the totality $\{H_k\}$ of these spaces is called the Sobolev tower  \cite{eng}.
In view of the definition, it is easy to show that the norm in  $H_{-k}$ ($k>0$) coincides with the norm
$$
\|u\|_{H_{-k}}=\sup_{v\in D((L^*)^k)}{|(u,v)|}/{\|v\|_{D(((L^*-\overline{\lambda} I)^k)^k)}}.
$$
The interpolation spaces $(H_{m},H_n)_{\theta,p}$ are described in
 \cite{gri1} under the additional condition that the operator  $L$ is positive, i.~e.,
$\{x\in {\mathbb R}:x\leq 0\}\subset \rho(L)$ and the inequality
\begin{equation}\label{eq111}
\|(L-\lambda I)^{-1}\|\leq c/(1+|\lambda|)\ \ \forall \lambda\leq 0
\end{equation}
holds. Equivalent norms in these spaces can be found, for instance, in
  \cite[Sect. 1.14.3]{tr01}.
Other classes of spaces  constructed with the use of a sectorial operator $L$
 are  described in \cite{haase} (see also the bibliography in  \cite{haase} and \cite{ausher}).
Their interpolation properties and, in particular, one more class of equivalent norms are also presented there.
Recall that  $L:H\to H$ is a sectorial operator if there exists $\theta\in [0,\pi)$ such that
$\sigma(L)\subset \overline{S_{\theta}}$ ($S_{\theta}=\{z: |{\rm arg\,}z|<\theta\}$),
${\mathbb C}\setminus \overline{S_{\theta}}\subset \rho(L)$, and, for every
 $\omega>\theta$, there exists a constant  $c(\omega)$ such that
\begin{equation}\label{eq112}
\|(L-\lambda I)^{-1}\|\leq c/|\lambda|\ \ \forall \lambda \in {\mathbb C}\setminus S_{\omega}.
\end{equation}
Let  $L$ be sectorial and injective (we do not require that $0\in \rho(L)$).
An analog of the space  $H_1$ in this case is the space
 $D_L$ which is a completion of  $D(L)$ with respect to the norm
 $\|Lu\|=\|u\|_{D_L}$ and an analog of  $H_{-1}$ is a completion of  $R(L)$ with respect to the norm $\|L^{-1}u\|=\|u\|_{R_L}$.
If $0\in \rho(L)$ then  $H_1=D_L $ and $H_{-1}=R_L$, otherwise these equalities are not true.

Let  $L:H\to H$ be an m-$J$-dissipative operator in the Krein space  $H$ with the indefinite inner product
$[\cdot,\cdot]=(J\cdot,\cdot)$, where $J$ is the fundamental symmetry and the symbol $(\cdot,\cdot)$ designates
the inner product  in $H$.
Define the space $F_1$ as the completion of   $D(L)$ with respect to the norm
$$
\|u\|_{F_1}^2=-{\rm Re\,}[Lu,u]+\|u\|^2, \ \ \|u\|=\|u\|_{H},
$$
and the space  $F_{-1}$ as the completion of  $H$ with respect to the norm
$$
\|u\|_{F_{-1}}=\sup_{v\in F_1}|[u,v]|/\|v\|_{F_1}.
$$
The space $F_{1}$ can be identified with a dense subspace of $H$ whenever
\begin{equation}\label{eq12}
\exists c>0 :\  |[Lu,v]|\leq c \|u\|_{F_1}\|v\|_{F_1}\ \ \forall u,v\in D(L).
\end{equation}
The proof can be found in Sect. 4, Chap. 1 in \cite{pyat2}.
It  also follows from  the arguments of
Sect.  7.3.2 in \cite{haase} (for example, from Propos.  7.3.4,
where it suffices to take  $A=-JL+I$).

It is a more or less obvious fact that
the condition (\ref{eq12}) is equivalent to the condition
\begin{equation}\label{eq13}
\exists c>0: \  |{\rm Im\,} [Lu,u]|\leq c\|u\|_{F_1}^2\ \  \forall u\in D(L).
\end{equation}
It suffices to take a sesquilinear form $a(u,v)=-[Lu,v]+(u,v)$ and apply Propos.  C.1.3 in \cite{haase}.

\begin{proposition}\label{pro2}
Let $L$ be an $m$-dissipative operator. Then  $-L$ is sectorial with $\theta=\pi/2$ and
\begin{equation}\label{eq122}
(H_1,H_{-1})_{1/2,2}=H.
\end{equation}
If additionally it is injective then
\begin{equation}\label{eq121}
(D_L,R_L)_{1/2,2}=H.
\end{equation}
\end{proposition}

The sectoriality with  $\theta=\pi/2$ results from  \cite[Sect. 7.1.1]{haase}).
The equality (\ref{eq121}) follows from Theorems  2.2 and 4.2 in \cite{ausher}
and the arguments after Theorem 2.2 (see also Sect. 7.3.1, Theorem 7.3.1, and the equality  (7.18) in
\cite{haase}). Take the operator $L-\varepsilon I$ ($\varepsilon>0$) instead of  $L$. This operator
is  $m$-dissipative and
$0\in \rho(-L+\varepsilon I)$. Hence, the equality (\ref{eq122}) results from  (\ref{eq121}).

Given a pair $H_1,H$ of Hilbert spaces and $H_1$ is densely embedded into $H$, by the negative space $H_1'$ constructed on this pair
we mean the completion of $H$ with respect to the norm
$$
\|u\|_{H_1'}=\sup_{v\in H_1}|(u,v)|/\|v\|_{H_1},
$$
where the brackets $(\cdot,\cdot)$ denote the inner product in $H$.
In this case the following known assertion holds (see  \cite[Chap. 1]{bere} and  \cite[the equality (2.2)]{gri2}).

\begin{proposition}\label{pro3}
The space of antilinear continuous functions over  $H_1$ can be identified with
 $H_{1}'$, the norm in   $H_1$ is equivalent to the norm
$\sup_{v\in H_{1}'}|(v,u)|/\|v\|_{H_{1}'}$, and
\begin{equation}\label{eq125}
(H_1,H_{1}')_{1/2,2}=H.
\end{equation}
\end{proposition}

\begin{lemma}\label{lem1}
{\it Let  $L:H\to H$ be an $m$-$J$-dissipative operator satisfying the condition {\rm (\ref{eq12})}.
Then
\begin{equation}\label{eq14}
(H_1,H_{-1})_{1/2,2}=H
\end{equation}
if and only if
\begin{equation}\label{eq15}
(F_1,F_{-1})_{1/2,2}=H.
\end{equation}}
\end{lemma}

\begin{proof}
Since $L$ is  $m$-$J$-dissipative, the operator $JL$ is   $m$-dissipative in  $H$.
In this case the operator  $JL-\varepsilon I$  is $m$-dissipative (the property  2 of Proposition \ref{pro1})
for every  $\varepsilon\geq 0$. Thus, the operator  $L-\varepsilon J$ and  the operator
$L_0=L-J$ as well is  $m$-$J$-dissipative and   $D(L_0)= D(L)$.
Moreover, the operator
$L_0$ is uniformly  $J$-dissipative and
\begin{equation}\label{eq16}
-{\rm  Re\, }[L_0u,u]= \|u\|_{F_1}^2\ \ \forall u\in D(L).
\end{equation}
In this case the property  5 of Proposition \ref{pro1} yields $i{\mathbb R}\in \rho(L_0)$.
Write out  the norm in  $H_{-1}$ for $u\in H$. We have that
$$
\|u\|_{H_{-1}}=\sup_{v\in D(L^*)}\frac{|(u,v)|}{\|v\|_{D(L^*)}}=
\sup_{v\in D(L^*)}\frac{|((L-J)^{-1}u,(L^*-J)v)|}{\|v\|_{D(L^*)}}.
$$
Using the fact that the norms  $\|(L^*-J)v\|$ and  $\|v\|_{D(L^*)}$ are equivalent and the previous equality, we obtain
the estimate  $\|u\|_{H_{-1}}\leq c\|(L-J)^{-1}u\|\leq c_1 \|u\|_{H_{-1}}$ valid for all $u\in H$ and some constants
 $c,c_1$.
Thus, we can introduce  the equivalent norm  $\|(L-J)^{-1}u\|$ in   $H_{-1}$.
The operator  $L_0$ satisfies the conditions of Lemma 4.2 of Chap.1 in \cite{pyat2}  (or Lemma  4.6 in \cite{pyat3}),
which implies the claim.
\end{proof}

\begin{lemma}\label{lem2}
{ Let  $L:H\to H$ be an $m$-$J$-dissipative operator satisfying the condition
 {\rm (\ref{eq12})}. Then  $L$ is extensible to an operator  $\tilde{L}$ of the class $L(F_1,F_{-1})$ and
\begin{equation}\label{eq22}
D(L)=\{u\in F_1: \tilde{L}u\in H\}.
\end{equation}
The operator   $L^c$  $(the J$-adjoint operator to $L)$
 is an $m$-$J$-dissipative operator satisfying
 {\rm (\ref{eq12})} and   $D(L^c)\subset F_1$ and this embedding is dense.}
\end{lemma}

 \begin{proof} The condition  (\ref{eq12}) ensures the estimate
\begin{equation}\label{eq21}
\|Lu\|_{F_{-1}}\leq c\|u\|_{F_1}.
\end{equation}
This estimate and the density of  $D(L)$ in  $F_1$ imply that the operator
 $L$ is extensible to an operator   $\tilde{L}\in L(F_1,F_{-1})$ and we have
the inequality  $-{\rm Re \,}[\tilde{L}u,u]\geq 0$ for all  $u\in F_1$. Now the equality
(\ref{eq22}) follows from the maximal  $J$-dissipativity of $L$.
Indeed, let  $D_0=\{u\in F_1: \tilde{L}u\in H\}$. Obviously,
 $D(L)\subset D_0$. The operator  $\tilde{L}:H\to H$ with the domain  $D_0$ is $J$-dissipative and an extension of $L$.
 In view of the maximal
 $J$-dissipativity of  $L$ we conclude that  $\tilde{L}=L$ and thereby  (\ref{eq22}) holds.
Consider the operator   $L_0=L-J$. As it was noted in the previous lemma,
the operator $L_0$ is maximal uniformly  $J$-dissipative with $D(L_0)=D(L)$. By Proposition 1 (the property 5),
${i{\mathbb R}}\in \rho(L_0)$.
Demonstrate that  $\tilde{L}_0=\tilde{L}-J$ is an isomorphism of $F_1$ onto  $F_{-1}$.
Indeed, if  $f\in F_{-1}$ then the expression  $[f,v]$ is an antilinear continuous functional
over $F_1$ and the Lax-Milgram theorem  (see, for instance,  \cite[Theorem  C.5.3]{haase}) implies that
there exists $u\in F_1$ such that
$
[\tilde{L}_0u,v]=[f,v]
$
for all  $v\in F_{1}$.
 Below, we use the same symbol $L_0$ for the operators  $L_0$ and $\tilde{L}_0$.
As is known (see  \cite[Chapt. 2, Sect. 1]{azi}, $(L_0^{-1})^c=(L_0^c)^{-1}$.
Note that the negative space  $F_1'$ constructed on the pair $F_1$ and $H$
is a completion of
$H$ with respect to the norm  $\|Ju\|_{F_{-1}}$ and, hence,  $J$ is an isomorphism of  $F_{-1}$ onto $F_1'$ with $J^{-1}=J$.
Proposition  3 implies that
every antilinear continuous functional over
 $F_1$ can be given in the form  $l(v)=[f,v]$, $f\in F_{-1}$, and every linear continuous functional
 over $F_{-1}$ in the form  $l(f)=[f,v]$, $v\in F_1$. Furthermore, the norm in $F_1$ is equiavlent to the norm
$
\|v\|_{F_1}=\sup_{f\in F_{-1}}|[f,v]|/\|f\|_{F_{-1}}.
$
We have that
\begin{equation}\label{eq211}
[L_0^{-1}u,v]=[u,(L_0^c)^{-1}v], \ \  \forall u,v\in H.
\end{equation}
Since  $L_0$ is an isomorphism of  $F_1$ and $F_{-1}$, the above equality
yields  $(L_0^c)^{-1}v\in F_{1}$. Indeed,
the left-hand side in (\ref{eq211}) is a linear continuous functional  $l(u)$ over $F_{-1}$. Hence, there exists
an element  $g\in F_1$ such that
\begin{equation}
[L_0^{-1}u,v]=[u,g], \ \  \forall u\in F_{-1}.
\end{equation}
In view of  (\ref{eq211}) we infer  $[u,g]=[u,(L_0^c)^{-1}v]$ for all  $u\in H$.
Therefore,  $g=(L_0^c)^{-1}v\in F_1$. Since every element
in  $D(L^c)$ is representable in the form  $(L_0^c)^{-1}v$, $v\in H$, we have that
$D(L_0^c)\subset F_1$ and the equality  (\ref{eq211}) implies that
$D(L_0^c)$ is dense in  $F_1$. Otherwise, there exists  $u\in F_{-1}$ such  that the right-hand side in (\ref{eq211})
vanishes for all  $v\in H$. In this case so does the left-hand side and thus $L_0^{-1}u=0$, i.~e.,  $u=0$.
\end{proof}

\begin{lemma}\label{lem3}
 { Assume that $L:H\to H$ is an  $m$-$J$-dissipative operator,
the condition  {\rm (\ref{eq12})} holds, and there exists a constant  $m>0$ such that
\begin{equation}\label{eq18}
\|u\|^2\leq m(-\Re[Lu,u]+\|u\|_{F_{-1}}^2)\ \ \forall  u\in D(L).
\end{equation}
Then there exists a number $\omega_0\geq 0$ such that

{\rm a)} $I_{\omega_0}=\{i\omega:\ |\omega|\geq \omega_0\}\subset \rho(L)\}$;

{\rm b)} if $i\omega\in \rho(L)$ $(\omega\in {\mathbb R})$ then the operator  $L-i\omega I:F_1\to F_{-1}$ is
continuously invertible;

{\rm c)} there exists a constant   $c>0$ such that
\begin{equation}\label{eq19}
\|(L-i\omega I)^{-1}u\|_{F_{-1}}\leq c\|u\|_{F_{-1}}/(1+|\omega|)\ \ \forall \omega\in  I_{\omega_0},\ \forall u\in F_{-1};
\end{equation}

{\rm d)} if   $i{\mathbb  R}\subset \rho(L)$ then there exists a constant
 $c>0$ such that the estimate \eqref{eq19} is true for all  $\omega\in {\mathbb R}$;

{\rm e)} if one of the conditions  {\rm  (\ref{eq14}) and (\ref{eq15})} holds then there exist a constant
 $c>0$ such that
\begin{equation}\label{eq20}
\|(L-i\omega I)^{-1}u\|\leq c\|u\|/(1+|\omega|)\ \ \forall \omega\in I_{\omega_0},\ \ \forall u\in H,
\end{equation}
and if in addition  $i{\mathbb  R}\subset \rho(L)$ then
\begin{equation}\label{eq200}
\exists c>0:\ \  \|(L-i\omega I)^{-1}u\|\leq c\|u\|/(1+|\omega|)\ \ \forall
\omega\in {\mathbb R},\ \ \forall u\in H.
\end{equation} }
\end{lemma}

\begin{proof}
Examine the equation
\begin{equation}\label{eq23}
Lu-i\omega u =f,\ \ \omega\in {\mathbb R},
\end{equation}
where $u\in D(L)$.
As a consequence, we have the equality
$
-{\rm Re\,} [Lu,u]=-{\rm Re\,}[f,u],
$
which in view of  (\ref{eq18}) can be rewritten as
$$
-{\rm Re\,} [Lu,u] + \delta\|u\|^2 \leq -{\rm Re\,}[f,u]-\delta m{\rm Re\,} [Lu,u]+ \delta m \|u\|_{F_{-1}}^2\ \  (\delta>0).
$$
Choosing  $\delta=1/(1+m)$, we infer
\begin{equation}\label{eq24}
\|u\|_{F_1}^2 \leq -(m+1){\rm Re\,}[f,u]+  m \|u\|_{F_{-1}}^2.
\end{equation}
From (\ref{eq21}) and  (\ref{eq23}) it follows that
\begin{equation}\label{eq25}
|\omega|^2\|u\|_{F_{-1}}^2\leq 2c^2\|u\|_{F_{1}}^2+2\|f\|_{F_{-1}}^2.
\end{equation}
Dividing this inequality by  $4c^2$ and adding it to the previous inequality, we derive
the inequality
\begin{equation}\label{eq261}
\frac{1}{2}\|u\|_{F_1}^2 +\delta_0|\omega|^2\|u\|_{F_{-1}}^2\leq
c_2\|f\|_{F_{-1}}^2-(m+1){\rm Re\,}[f,u]+m\|u\|_{F_{-1}}^2\ \ (\delta_0>0).
\end{equation}
The Schwartz inequality $|[f,u]|\leq \|f\|_{F_{-1}}\|u\|_{F_{1}}$ holds.
Using this inequality in the right-hand side of  (\ref{eq261}) and the inequality
\begin{equation}\label{cau}
|ab|\leq \varepsilon |a|^p/p+ |b|^q/(q\varepsilon^{q/p}),\ \ \varepsilon\in (0,\infty),\  1/p+1/q=1, \ p\in (1,\infty),
\end{equation}
where  $p=q=2$, we infer
\begin{equation}\label{eq26}
\frac{1}{2}\|u\|_{F_1}^2 +\delta_0|\omega|^2\|u\|_{F_{-1}}^2\leq c_3(\varepsilon)\|f\|_{F_{-1}}^2+2\varepsilon\|u\|_{F_{1}}^2 +
m\|u\|_{F_{-1}}^2.
\end{equation}
Choosing $\varepsilon=1/8$ and transferring the summand with $\|u\|_{F_1}$ into the left-hand side we arrive at
the inequality
\begin{equation}\label{eq27}
\|u\|_{F_1}^2 +|\omega|^2\|u\|_{F_{-1}}^2\leq c_4\|f\|_{F_{-1}}^2+c_5\|u\|_{F_{-1}}^2,
\end{equation}
where  $c_4,c_5$ are some constants independent of  $u$.
Putting  $\omega_0=2\sqrt{c_5}+1$, we arrive at the inequality  (\ref{eq19}) which implies that
${\rm ker\,}(L-i\omega I)=\{0\}$ for $\omega\in I_{\omega_0}$. The property  6 of Proposition  \ref{pro1} implies that
the subspace  $R(L-i\omega I)$ is dense in  $H$ and thus in $F_{-1}$.
The estimate  (\ref{eq19}) implies that the operator  $(L-i\omega I)^{-1}$ defined on a dense subspace of $F_{-1}$
admits an extension by continuity on the whole  $F_{-1}$. Therefore,  the operator $\tilde{L}-i\omega I$ is an isomorphism
of $F_{1}$ onto  $F_{-1}$, where  $\tilde{L}:F_1\to F_{-1}$ is an extension of $L$ to a continuous map of  $F_1$ into $F_{-1}$.
The equality  (\ref{eq22}) ensures the inclusion  $i\omega\in \rho(L)$ for all  $\omega\in I_{\omega_0}$.
Let $i\omega\in \rho(L)$ ($\omega\in {\mathbb  R}$).
 Take an arbitrary number  $\omega_1\in I_{\omega_0}$.
In view of (\ref{eq11}) we have that
$$
(L-i\omega I)^{-1}=(L-i\omega_1 I)^{-1}-i(\omega_1-\omega)(L-i\omega I)^{-1}(L-i\omega_1)^{-1}
$$
In this case, for $f\in H$, we have
$$
\|(L-i\omega I)^{-1}f\|_{F_1}\leq \|(L-i\omega_1 I)^{-1}f\|_{F_1}+(|\omega|+|\omega_1|)\|(L-i\omega I)^{-1}(L-i\omega_1 I)^{-1}\|_{F_1}.
$$
Since  $0\in \rho(L-i\omega I)$, there exists a constant  $c>0$ such that  $\|u\|_{F_1}\leq c\|(L-i\omega I)u\|$ for all  $u\in D(L)$.
Using this inequality and the embedding  $F_1\subset H$, we obtain that
$$
\|(L-i\omega I)^{-1}f\|_{F_1}\leq \|(L-i\omega_1 I)^{-1}f\|_{F_1}+c(|\omega|+|\omega_1|)\|(L-i\omega_1 I)^{-1}\|_{F_1}.
$$
Since  $|\omega_1|\geq \omega_0$, the right-hand side can be estimated by
$c\|f\|_{F_{-1}}$. Thus, we arrive at the estimate
$$
\|(L-i\omega I)^{-1}f\|_{F_1}\leq c\|f\|_{F_{-1}}\ \  \forall f\in R(L-i\omega I).
$$
As before,  we establish that the operator  $\tilde{L}-i\omega I:F_1\to F_{-1}$ is continuously invertible.
The statement  d) easily follows from  b) and c). The statement e) follows from interpolation arguments.
In what follows, we identify  $L$ and its extension $\tilde{L}$.
The operator  $(L-i\omega I)^{-1}$ ($\omega\in I_{\omega_0}$) is a continuous mapping of
 $F_{-1}$ into  $F_{-1}$ satisfying the inequality
\begin{equation}\label{eq28}
\|(L-i\omega I)^{-1}u\|_{F_{-1}}\leq c\|u\|_{F_{-1}}/(1+|\omega|)\ \ \forall \omega\in I_{\omega_0}
\end{equation}
for some constant  $c>0$ and all ` $u\in F_{-1}$. Take  $u\in F_{1}$.
In this case,  $g=Lu-i\omega_0 u\in F_{-1}$ and   $(L-i\omega I)^{-1}g\in F_{1}$; thereby
$(L-i\omega_0 I)(L-i\omega I)^{-1}u=(L-i\omega I)^{-1}g\in F_1$ and
\begin{equation}\label{eq29}
 \|(L-i\omega I)^{-1}u\|_{F_{1}}\leq c \|(L-i\omega I)^{-1}g\|_{F_{-1}}\leq
 \frac{ c_1\|g\|_{F_{-1}}}{1+|\omega|}\leq  \frac{c_2\|u\|_{F_1}}{1+|\omega|}
\end{equation}
for all   $ \omega\in I_{\omega_0}$ and           $u\in F_{1}$.
Thus,  $(L-i\omega I)^{-1}|_{F_1}\in L(F_1)$.
The conventional properties of interpolation spaces  (see, for instance, the claim a) of Theorem 1.3.3 in \cite{tr01})
yield $(L-i\omega I)^{-1}|_{(F_1,F_{-1})_{1/2,2}}\in L((F_1,F_{-1})_{1/2,2})$  and the estimates  (\ref{eq28}), (\ref{eq29})
and the equality  (\ref{eq15}) ensure  the estimates  (\ref{eq20}). The last statement is obvious.
\end{proof}

\begin{lemma}\label{lem4}
 {Assume that  $H$ is a Krein space and $L:H\to H$ is an $m$-$J$-dissipative operator such that

{\rm 1)} there exist maximal nonnegative and nonpositive
subspaces  $M^{\pm}$ invariant under  $L$ such that  $H=M^+ + M^-$ {\rm (}the sum is direct{\rm )};

{\rm 2)} the subspace  $(M^+\cap D(L))+(M^-\cap D(L))$ is dense in $D(L)$.

Then if  $i\omega\in \rho(L)$ $(\omega\in {\mathbb R})$ then  $i\omega \in \rho(L|_{M^{\pm}})$.}
\end{lemma}

\begin{proof}
Put $H_1^{\pm}=M^{\pm}\cap D(L)$ and $H_1=D(L)$, where the graph norm  is introduced.
Without loss of generality, we can assume that  $\omega=0$;
otherwise, we consider the operator  $L-i\omega I$. By condition,
 $L(H_1^{\pm})\subset M^{\pm}$. The condition  1) implies that the subspaces
 $M^{\pm}$ are closed subspaces of  $H$. By definition,
$H_1^{\pm}$ are closed subspace of  $H_1$. Since  $0\in \rho(L)$, $L(H_1^{\pm})$ are closed subspaces
of  $H$. In view of  2), the subspace $L(H_1^+ + H_1^-)$ is dense in $H$. But in this case
the subspaces $L(H_1^{\pm})$ are dense in  $M^{\pm}$. Since they are closed,  $L(H_1^{\pm})=M^{\pm}$. The inverse mapping theorem implies
that  $0\in  \rho(L|_{M^{\pm}})$.
\end{proof}

\begin{lemma}\label{lem5}
{  Let $H$ be a Krein space and let $L:H\to H$ be an  $m$-$J$-dissipative operator
satisfying the conditions  {\rm (\ref{eq12}) and (\ref{eq18})}. Then there exist maximal nonnegative and nonpositive subspaces
 $M^{\pm}$ invariant under  $L$ such that
$\overline{C^{\pm}}\subset \rho(L|_{M^{\pm}})$,  $H=M^+ + M^-, $
where the sum is direct, and  $C^{\pm}=\{z\in {\mathbb C}:\ \pm\textrm{Re\,}z>0\}$. Denote by
$F_{-1}^\pm$ the completions of  $M^\pm$ with respect to the norm in  $F_{-1}$.
Then there exists a constant   $c>0$ such that
\begin{equation}\label{eq191}
\|(L-zI)^{-1}u\|_{F_{-1}}\leq c\|u\|_{F_{-1}}/(1+|z|)\ \ \forall z\in \overline{C^{\pm}}, \ \   \forall u\in F_{-1}^\pm,
\end{equation}
where  $L$ denotes  the extension of  $L$ to an operator of the class  $L(F_1,F_{-1})$.
The operators  $\pm L|_{F_{-1}^{\pm}}$ treated as operators from  $F_{-1}^{\pm}$ into $F_{-1}^{\pm}$
with domains  $F_1^\pm=F_1\cap M^{\pm}$ are generators of analytic semigroups.}
\end{lemma}

\begin{proof} By Lemma  \ref{lem3}, there exits a constant  $\omega_0>0 $ such that  $I_{\omega_0}\in \rho(L)$. Using the embedding
 $i{\mathbb R}\subset \rho(L|_{M^{\pm}})$,  we obtain that  $D(L)=H_1^{+}+ H_1^-$, with $H_1^{\pm}=D(L)\cap M^{\pm}$.
This equality yields  $i{\mathbb R}\in \rho(L)$. Lemma  \ref{lem3} implies that
the estimate  \eqref{eq191} is valid for all  $z$ lying on the imaginary axis. To obtain the estimate
\eqref{eq191}, we just repeat the arguments of Lemma \ref{lem3} for the operators  $L|_{M^{\pm}}$. To validate
the estimate, we employ the inequality
 ${\rm Re\,}z[u,u]\geq 0$ for all  $u\in M^{\pm}$ and $z$ with $\pm {\rm Re\,}z\geq 0$.
The last assertion of the lemma follows from the conventional properties of
analytic semigroups   (see \cite[Chap. 2, Theorem  4.6 (e)]{eng}).
\end{proof}

\section{Main results.}

The most part of the statements of the following theorem is obtained in
\cite{pyat1} (see also  \cite[Chap.1, Theorem 4.1]{pyat2}.

\begin{theorem}\label{th1}
 { Let $L:H\to H$ be an $m$-$J$-dissipative operator in a Krein space  $H$ such that
  $i{\mathbb R}\in \rho(L)$ and the conditions  {\rm (\ref{eq14})}  and   {\rm (\ref{eq200})} hold.
Then there exist maximal nonnegative and maximal nonpositive subspaces  $H^\pm$ invariant under
 $L$. The whole space  $H$ is representable as the direct sum  $H  = H^+  + H^-$,
$\sigma(\mp L|_{H^\pm})\subset {\mathbb C}^\pm$, and the operators  $\pm L|_{H^\pm}$ are generators of
analytic semigroups.}
\end{theorem}

The last claim of the theorem that  $L|_{H^\pm}$ are generators of analytic semigroups is not contained
in the corresponding theorem in \cite{pyat1}. However, it is a consequence of this theorem
and it was actually proven in \cite{gri2}
(see  \cite[Corollary 3.3 ]{gri2}).

In the following theorem we refine Theorem  4.2  of \cite[Chap. 1]{pyat2} (see also \cite{pyat1}); in contrast to
\cite{pyat1,pyat2}, we do not require the uniform $J$-dissipativity of  $L$ and prove some new statements.

\begin{theorem}\label{th2}
{ Let  $L:H\to H$ be an  $m$-$J$-dissipative operator
satisfying the conditions  $i{\mathbb R} \subset \rho(L)$,  {\rm (\ref{eq12})}, and {\rm (\ref{eq14})}.
Then there exist maximal nonnegative and maximal nonpositive subspaces $H^\pm$ invariant under
 $L$. The whole space  $H$ is representable as the direct sum   $H  = H^+  + H^-$,
$\sigma(\mp L|_{H^\pm})\subset {\mathbb C}^\pm$, and the operators $\pm L|_{H^\pm}$ are generators of analytic semigroups.
If in addition  $L$ is strictly  $J$-dissipative or uniformly  $J$-dissipative then
the corresponding subspaces $H^+$ and $H^-$ can be chosen to be positive and negative or
uniformly positive and uniformly negative, respectively.}
\end{theorem}

\begin{proof}
By Lemma  \ref{lem3}, we have the estimate
(\ref{eq200}) (note that the condition (\ref{eq18}) is fulfilled due to the equality (\ref{eq15}) of Lemma  \ref{lem1}).
The operators $P^{\pm}$ of the parallel projection onto  $H^{\pm}$ (see  \cite{pyat1}-\cite{pyat3}) such that
$H=P^+H+ P^-H$, $P^+P^-=P^-P^+=0$ are constructed as follows.
As is known, there exits a constant
 $\delta > 0$  such that
$$
S={\mathbb C} \setminus (S^+ \cup S^- ) \subset \rho(L),
$$
$$
S^{+(-)}  =\{\lambda : |{\rm arg}\,\lambda| < \pi/2 - \delta\  (|{\rm arg}\,\lambda| >\pi/2  + \delta), |z|> \delta \}
$$
and, for every ray  ${{\rm arg}\,\lambda = \theta}\subset  S$, the resolvent estimate
\begin{equation}\label{eq30}
\|(L - \lambda)^{-1}f \| \leq c_1 \|f\| (1+ |\lambda|)^{-1}
\end{equation}
holds.
We have that
$$
P^{\pm}f =- \frac{1}{2\pi i} \int_{\Gamma^{\pm}} \frac{L(L+\lambda)^{-1}}{\lambda}  f \,dz,\
\Gamma^{\pm}  =\partial S^{\pm},\  f \in  D(L),
$$
where the integration over $\Gamma^{\pm}$ is taken in the positive direction
with respect to the domains  $S^{\pm}$. As it is easy to see, the integral are normally convergent for  $f\in D(L)$
and thus the quantities  $P^{\pm}f$ are defined correctly.
Lemma  \ref{lem1} implies that
the operators $P^{\pm}$ are extensible to operators of the class
$L(H)$. Indeed, as a consequence of Theorem 1.15.2 (see also Sect. 1.15.4) in \cite{tr01}  and the reiteration theorem
 (see \cite{tr01})
\begin{equation}\label{eq31}\begin{array}{c}
(H_1,H)_{\theta,2}=(H_1,H_{-1})_{\theta/2,2},\  (H,H_{-1})_{1-\theta,2}=(H_1,H_{-1})_{1-\theta/2,2},\ \vspace{8pt} \\
H=(H_1,H_{-1})_{1/2,2}=((H_1,H_{-1})_{\theta/2,2},(H_1,H_{-1})_{1-\theta/2,2})_{1/2,2}\end{array}
\end{equation}
The former part of the last equality is the condition (\ref{eq14}).
At the same time   the projections $P^{\pm}$ are extensible to operators of the class
$L((H_1,H)_{\theta,2})$, $L((H,H_{-1})_{1-\theta,2})$ (see  \cite[Sect. 3]{gri2}).
In this case the last equality in (\ref{eq31}) ensures the claim.
Moreover, using the  definitions,  we establish that
 $(P^+ + P^-)f = f$ for all  $f\in D(L)$ and thus for all   $f \in H$.
Thereby the space   $H$ is representable as the direct sum  $H  = H^+   + H^-$, with
$ H^{\pm}  = \{u \in H  : P^{\pm} u = u \}$.
Demonstrate that  $H^+$ and $H^-$ are nonnegative and nonpositive subspaces, respectively. The proof is similar to
the corresponding proof in   \cite{pyat1,pyat2}.
For example, we consider  $H^+$. In view of the density
of $D(L^k )$ in $H$  for all  $k>0$ (\cite[Sect. 1.14.1]{tr01}) and boundedness of
 $P^+$, for every  $u \in  H^+$ there exists a sequence
$u_n \in   D(L^2 ) \cap H^+ $ such that  $\|u_n  - u \| \to 0$  as  $ n \to \infty$.
Indeed, the operators $P^{\pm}$ and $L$ commute in the sense that
if  $f \in D(L^k)$ ($k\geq 1$) then $P^{\pm}f \in D(L^k)$ and $L^k P^{\pm} f = P^{\pm} L^k f$.
Find a sequence  $v_n \in D(L^2)$ such that
$\| v_n - u \| \to 0 $  as $n \to \infty$. Put
$u_n = P^+ v_n \in D(L^2) \cap H^+$. In this case
$\| u_n - u \| = \|P^+(v_n - u) \| \to 0$ as $n \to \infty$.
Define an operator
\begin{equation}\label{eq32}
P u =- \frac{1}{2\pi i} \int_{\Gamma^{+}} e^{-\lambda t} \frac{L(L + \lambda)^{-1} }{\lambda}  u \,d\lambda,\
t> 0.
\end{equation}
Take  $v_n(t) = Pu_n$. Using  (\ref{eq32}), we infer
\begin{equation}\label{eq33}
\|v_n\|_{L_{2}(0,\infty;H)} \leq
 \frac{1}{2\pi } \int_{\Gamma^{+}}
 \frac{c \|u_n\|}{|\lambda| \sqrt{\Re\,\lambda}}\, |d \lambda| \leq c_1 \|u_n \|,
\end{equation}
where the constant  $c$ is taken from the inequality
$
\|L(L+\lambda)^{-1} u_n \| \leq c \|u_n\|,\ \
\Gamma^{\pm}  =\partial S^{\pm}.
$
Employing the normal convergence of the integrals obtained
from (\ref{eq32}) by the formal differentiation with respect to
 $t$ and the formal application of $L$, we can say that
$v_n(t) \in L_2(0,\infty;F_1)$ and the distributional derivative
 $v'(t)$ possesses the property
$v_n'(t) \in L_2(0,\infty;H)$. Hence, after a possible modification on a set of zero measure
 $v_n(t) \in C ([0,\infty);H)$, i.~e., $v_n(t)$ is a continuous function with values in  $H$.
It is immediate that
\begin{equation}\label{eq34}
v_n(0) = u_n,
\end{equation}
\begin{equation}\label{eq35}
S v_n = v_n' - L v_n = 0.
\end{equation}
Using the equalities  (\ref{eq34}) and  (\ref{eq35}) and integrating by parts, we infer
\begin{equation}\label{eq36}
0 = \Re\,\int_{0}^{\infty} [Sv_n(t), v_n(t)] \,dt
= -[v_n(0),v_n(0)]_0 - \int_{0}^{\infty} \Re\,[Lv_n,v_n](\tau)\,d\tau.
\end{equation}
Therefore, we conclude that
\begin{equation}\label{eq37}
[u_n,u_n]_0 = -\int_{0}^{\infty} {\rm Re\,}[Lv_n,v_n](\tau)\,d\tau,
\end{equation}
i.~e., $[u_n,u_n] \geq 0$. Passing to the limit on  $n$ in this inequality, we derive  $[u,u]\geq 0$, i.~e., $H^+$ is a nonegative
subpace. By analogy, we can prove that  $H^-$ is a nonpositive subspace.
Prove the maximality of
$H^+$. Assume the contrary that there exists a nonnegative subspace  $M$ such that
$H^+ \subset M$ and $H^+ \neq M$. Take an element
 $\varphi \in M\setminus H^+$. In this case  $\varphi+ \psi\in M$ and
$[\varphi+ \psi, \varphi+ \psi]_0 \geq 0$ for all  $ \psi \in H^+$. Represent  $\varphi$ in the form $\varphi =\varphi_1 + \varphi_2$,
$\varphi_1 \in H^+$ , $\varphi_2 \in H^-$. Taking  $\psi =- \varphi_1$, we obtain that  $\varphi_2\in M$ and
$[\varphi_2 , \varphi_2 ]_0 \geq 0$. But  $\varphi_2 \in  H^-$. Hence, $[\varphi_2 ,\varphi_2]_0=0$.
Since  $[\varphi,\varphi]\geq 0$ for all  $\varphi\in M$ and
$-[\varphi,\varphi]\geq 0$ for all  $\varphi\in H^-$,  the Cauchy-Bunyakovski\u{i} inequality
\begin{equation}\label{cauchy}
|[\varphi,\psi]|\leq |[\varphi,\varphi]|^{1/2}|[\psi,\psi]|^{1/2}
\end{equation}
holds for all  $\varphi,\psi\in M$ and all  $\varphi,\psi\in H^-$. This inequallity yields
$[\varphi_2, \psi]_0 = 0$ for all  $\psi \in H^+\subset M$ and all  $\psi\in H^-$.
Thus, we have that $\varphi_2=0$.
The maximality of  $H^-$ is proven by analogy. Now  we demonstrate that the subspaces $H^{\pm}$ are strictly or uniformly definite
whenever  $L$ is strictly or uniformly $J$-dissipative. Let  $L$ be a uniformly  $J$-dissipative operator.
In this case we can use the norm $\|u\|_{F_1}^2=-{\rm Re\,}[Lu,u]$ as an equivalent norm
in   $F_1$. In view of (\ref{eq35}) and (\ref{eq37}), we conclude that
\begin{equation}\label{eq38}
[u_n,u_n]\geq
\delta_0(\|v_n'\|_{L_2(0,\infty;F_{-1})}^2+\|v_n\|_{L_2(0,\infty;F_{1})}^2),
\end{equation}
where $\delta_0$ is a positive constant independent of  $n$.
Applying (\ref{eq38}) to the difference  $u_n-u_m$ and using the fact that
 $u_n$ is a Cauchy sequence in $H$, we derive that
the sequence  $v_n$ is a Cauchy sequence in  $L_2(0,\infty;F_1)$ and thereby it converges to some function
 $v(t)\in L_2(0,\infty;F_1)$ and also
$v_n'(t) \to v'(t)$ in $L_2(0,\infty;F_{-1})$. The trace theorem  (see  \cite[Theorem 1.8.3]{tr01}) and the inequality
 (\ref{eq15})  implies that $v_n(0)\to v(0)$ in  $H$ and there exists a constant
  $\delta>0$ such that
$$
\|v_n'\|_{L_2(0,\infty;F_{-1})}^2+\|v_n\|_{L_2(0,\infty;F_{-1})}^2\geq \delta\|u_n\|_{H}^2.
$$
This inequality and (\ref{eq38}) imply the estimate
$
[u_n,u_n]\geq \delta\|u_n\|_{H}^2/2.
$
Passing to the limit on $n$ in this estimate  we arrive at the inequality
$
[u,u]\geq \delta\|u\|_{H}^2/2
$
valid for all  $u\in H^+$ which ensures the uniform positivity  of $H^+$.
Similarly, we can prove that  $H^-$ is a unifrmly negative subspace.
Let  $L$ be a strictly  $J$-dissipative operator.
Demonstrate that the subspaces  $H^{\pm}$ are of definite sign.
Consider, for example, the subspace $H^+$.
Assume the contrary that this subspace is degenerate, i.~e.,    $H^+_0=H^+\cap (H^+)^{[\perp]}\neq \{0\}$,
where the symbol  $[\perp]$ stands for the orthogonal complement with respect to
an indefinite inner product.
We have that  $[\varphi,\varphi]=0$ for all  $\varphi\in H_0^+$.
The inequality  (\ref{cauchy}) implies that
$$
[\varphi,\psi]=0
$$
for all  $\varphi\in H_0^+$ and $\psi\in H^+$. Take $v\in H_0^+$. By construction,
the subspaces $H^{\pm}$ meet the conditions of Lemma \ref{lem4}.
Lemma \ref{lem4} implies that  $L^{-1}v\in H^+$ and thus
$
[L^{-1}v,v]=0
$
for all  $v\in H_0^{+}$. Since the operator $L^{-1}$ is also $J$-dissipative, for all such  $v$
$v\in {\rm ker\,}(L^{-1}+(L^{-1})^c)={\rm ker\,}(L^{-1}+(L^c)^{-1})$ (see \cite[Chap. 2, Sect. 1]{azi}).
Putting $u=L^{-1}v$, we infer
$u=-(L^{c})^{-1}Lu$ and thereby  $u\in D(L)\cap D(L^c)$.
Moreover,  $Lu+L^cu=v-v=0$. Hence, we have the equality  ${\rm Re\,}[Lu,u]=0$
which contradicts to the strict dissipativity of  $L$.
\end{proof}

Procced with necessary conditions. Actually, we will prove that the conditions
of Theorem  \ref{th2} are necessary and sufficient for its statement to be true.

\begin{theorem}\label{th3}
 { Assume that  $L:H\to H$ is an  $m$-$J$-dissipative operator such that

{\rm 1)} there exists  $\omega\in {\mathbb R}:\ i\omega\in \rho(L)$;

{\rm 2)} there exist maximal uniformly positive and uniformly negative
subspaces  $M^{\pm}$ invariant under  $L$ such that  $H=M^+ + M^-$ {\rm (}the sum is direct{\rm )};

{\rm 3)} the subspace  $(M^+\cap D(L))+(M^-\cap D(L))$ is dense in $D(L)$.

Then  the equality {\rm (\ref{eq14})} holds.}
\end{theorem}

\begin{proof}
Without loss of generality we may assume that  $\omega=0$.
Define the projections  $P^{\pm}$ onto  $M^{\pm}$, respectively, corresponding to the decomposition
$H=M^+ +M^-$. Thus,  $P^+ + P^- =I$, $P^+P^-=P^-P^+=0$, $P^{\pm}u=u\ $ for all   $u\in M^{\pm}$.
In view of Lemma \ref{lem4},  $L^{-1}\varphi\in M^{\pm}\cap H_1$  ($H_1=D(L)$) for  $\varphi\in M^{\pm}$. Hence, we have that
\begin{equation}\label{eq39}
P^{\pm}L^{-1}u=L^{-1}P^{\pm}u
\end{equation}
for all  $u\in H$. As a result,   $P^{\pm}|_{H_1}\in L(H_1)$,  the equalities
$P^{\pm}Lu=LP^{\pm}u$ are true for  $u\in H_1$, and  $H_1=H_{1}^+ +H_{1}^-$ ($H_{1}^{\pm}=R(P^{\pm})$).
The equalities  (\ref{eq39}) imply that  $P^{\pm}$ are extensible to operators
of the class $L(H_{-1})$ and as before the equality
 $P^++P^-=I$ yields
 $H_{-1}=H_{-1}^+ +H_{-1}^-$, $H_{-1}^{\pm}=R(P^{\pm})$.
In this case, we infer
$$
P^{\pm}\in L(H_{1-2\theta}) \ \forall \theta\in (0,1), \   H_{1-2\theta}=(H_1,H_{-1})_{\theta,2}
$$
and Theorem  1.17.1 in \cite{tr01} implies that
$$
H_{1-2\theta}^{\pm}=(H_1^{\pm},H_{-1}^{\pm})_{\theta,2}=\{u\in H_{1-2\theta}:\  P^{\pm}u=u\}=P^{\pm}H_{1-2\theta}.
$$
In particular, we have that
\begin{equation}\label{eq411}
H_{1-2\theta}=P^+H_{1-2\theta} + P^-H_{1-2\theta}=H_{1-2\theta}^{+} + H_{1-2\theta}^{-}\ \ \textrm{(the sum is direct)}.
\end{equation}
Let the symbol $[\cdot,\cdot]$ denote an indefinite inner product in  $H$. Since
the subspaces  $M^{\pm}$ are uniformly definite, the expression
$(u,v)_{\pm}=\pm [u,v]$ is an equivalen inner product on  $M^{\pm}$, respectively.
Consider the operator  $L|_{M^{+}}:M^{+}\to M^+$. In this new inner product the operator
$L|_{M^{+}}$ is $m$-dissipative. By Proposition  2,
$H_0^+=(H_1^+, H_{-1}^+)_{1/2,2}=M^+$. Similarly we have that $H_0^-=(H_1^-, H_{-1}^-)_{1/2,2}=M^-$.
In this case the claim results from the equality (\ref{eq411}) with $\theta=1/2$.
\end{proof}

\begin{theorem}
\label{th4}
 { Assume that   $L:H\to H$ is an $m$-$J$-dissipative operator satisfying the condition {\rm (\ref{eq12})} and
there exists a maximal uniformly positive {\rm (}or uniformly negative{\rm )} subspace
 $M$ invariant under  $L$ and such that  $i{\mathbb R}\cap \rho(L|_{M})\cap \rho(L)\neq \emptyset$.
Then the equality  {\rm (\ref{eq14})} holds.}
\end{theorem}

\begin{proof}
Assign $H_1^+=H_1\cap M$, $H_1=D(L)$,  $F_{1}^+=F_1\cap M$ (the spaces $F_1$, $F_{-1}$ were constructed before Proposition
 2). For definiteness, we assume that  $M$ is uniformly positive and $0\in \rho(L|_{M})\cap \rho(L)$.
Since $M$ is a maximal uniformly positive  subspace,  $M$ is projectively complete, i.~e.,
the whole space  $H$ is representable as the direct sum
 $H=M+N$, $N=M^{[\perp]}$. The subspace $N$ is uniformly negative  (see \cite[Chap. 1, Corollary 7.17]{azi}) and
the corresponding projection  $P$ onto $M$ parallel to  $N$ is $J$-selfadjoint.
Introduce the operator  $L_0=L-PJP-(I-P)J(I-P)$ taking  $D(L_0)=D(L)$. We have that
$$
-{\rm Re\,}[L_0u,u]=-{\rm Re\,}[Lu,u]+ \|Pu\|^2+\|(I-P)u\|^2\geq \|u\|_{F_1}^2/2.
$$
Thus, the quantity  $-{\rm Re\,}[L_0u,u]$ is the square of an equivalent norm  in  $F_1$.
The operator  $L$ meets the condition  (\ref{eq12}) and thereby, by Lemma  \ref{lem2}, the operator   $L$ admits an extension
 $\tilde{L}$ to an operator of the class  $L(F_1,F_{-1})$. Put
$\tilde{L}_0=\tilde{L} -PJP-(I-P)J(I-P)\in L(F_1,F_{-1})$. By construction,  $L_0$ also satifies
the condition (\ref{eq12}).  As in Lemma  \ref{lem2}, we can establish that
$\tilde{L}_0$ is an isomorphism of $F_1$ onto  $F_{-1}$. Demonstrate that  $0\in \rho(L_0)$.
Consider the restriction of $\tilde{L}_0$ on $D_0=\{u\in F_1: \  \tilde{L}_0 u\in H\}$. On the one hand,
 $D(L)\subset D_0$. Indeed, if $u\in D(L)$ then  $u\in F_1$, $Lu\in H$, $PJPu, (I-P)J(I-P)u\in H$ and thus  $u\in D_0$.
Let  $u\in D_0$. In this case  $u\in F_1$ and thereby  $PJPu, (I-P)J(I-P)u\in H$, and  $\tilde{L}u\in H$.
Therefore,  $D_0\subset \{u\in F_1:\ \tilde{L}u\in H\}=D(L)$ (see  (\ref{eq22})).
Hence, $D_0=D(L)$. Since  $\tilde{L}_0$ is an isomorphism of  $F_1$ onto $F_{-1}$, the definition
of the class  $D_0$ readily implies that
$0\in \rho(L_0)$. In  what follows, we write $L,L_0$ rather than  $\tilde{L}$ and $\tilde{L}_0$, respectively.
Put $F_1^+=F_1\cap M$.
Since  $0\in \rho(L|_{M})\cap \rho(L)$, the operator $L$ maps  $D(L)\cap M$ onto $M$.
For $u\in D(L)\cap M$, we have that
$$
\|u\|_{F_1}\leq \sqrt{2}({\rm Re\,}[-L_0u,u])^{1/2}\leq
2\sup_{v\in F_1^+}|[L_0u,v]|/\|v\|_{F_1}\leq
$$
$$
\leq 2\sup_{v\in F_1}|[L_0u,v]|/\|v\|_{F_1}=2\|L_0u\|_{F_{-1}}\leq 2c\|u\|_{F_1},
$$
where the constant $c$ is independent of  $u$.
Whence, we infer
\begin{equation}
\sup_{v\in F_1^+}|[L_0u,v]|/\|v\|_{F_1}\leq
\sup_{v\in F_1}|[L_0u,v]|/\|v\|_{F_1}\leq 2c\sup_{v\in F_1^+}|[L_0u,v]|/\|v\|_{F_1}\
\end{equation}
for all  $ u\in D(L)\cap M$.
In particular, we have the inequality
\begin{equation}\label{eq421}
\sup_{v\in F_1^+}|[u,v]|/\|v\|_{F_1}\leq
\sup_{v\in F_1}|[u,v]|/\|v\|_{F_1}\leq 2c\sup_{v\in F_1^+}|[u,v]|/\|v\|_{F_1}\  \ \forall u\in M.
\end{equation}
Denote by  $F_{-1}^+$ the completion of
$M$ with respect to the norm
 $\sup_{v\in F_1^+}|[u,v]|/\|v\|_{F_1}$. In view of (\ref{eq421}), this norm is equivalent to the usual norm
 of the space  $F_{-1}$ and thus the space  $F_{-1}^+$ can be identified with a closed subspace of
 $F_{-1}$. The operator $L_0$ is an isomorphism  $F_1^+$ onto $F_{-1}^+$.
In view of uniform positivity of  $M$, the expression
$[\cdot,\cdot]$ is an equivalent inner product on $M$. Then the space
$F_{-1}^+$ coincides with the negative space constructed on the pair
 $F_1^+,M$ and Proposition  3 yields
$$
(F_1^+,F_{-1}^+)_{1/2,2}=M.
$$
Demonstrate that the subspace $N$ is invariant under  $L^c$ which is
$m$-$J$-dissipative along with  $L$ and $0\in \rho(L^c)$ (see \cite[Chapt. 2,  Sect. 1, Theorem 1.16]{azi}). Lemma
\ref{lem2} implies that $D(L^c)$ is densely embeded into  $F_1$ and thus
 $L^c$ as well as the operator $L$ itself admits an extension to
 a continuous  operator from  $F_1$ into $F_{-1}$.
By Proposition  1.11  in Sect. 1 of Chap.2 \cite{azi}, $L^c(D(L^c)\cap N)\subset N$. Since
 $0\in \rho(L|_M)$, $L^{-1}M\subset M$ and we have that
 $(L^c)^{-1}N\subset N$ in view of the equality  $(L^{-1})^c=(L^c)^{-1}$
 (\cite[Chap. 2,   Sect. 1, Proposition  1.6]{azi}). These two containments and the inclusion
 $0\in \rho(L^c)$ ensure that  $0\in \rho(L^c|_N)$.
The density of  $D(L^c)\cap N$ in  $N$ is obvious. Indeed, assume the contrary. In this case there exists
an element  $\varphi\in D(L^c)\cap N$ such that
 $[\varphi,\psi]=0$ for all  $\psi\in N$. Assign
$\psi=(L^c)^{-1}v$, $v\in N$. In this case,  $0=[\varphi, (L^c)^{-1}v]=[L^{-1}\varphi, v]$ and thereby
$L^{-1}\varphi\in D(L)\cap M$, i.~e., $\varphi\in M$. Therefore, we have that  $\varphi\in M\cap N=\{0\}$.
Put  $F_{1}^-=F_1\cap N$. Denote by  $F_{-1}^-$ the completion of  $N$ with respect to the norm
$\sup_{v\in F_1^-}|[u,v]|/\|v\|_{F_1}$. Repeating the above arguments for the subspace  $M$,
we conclude that this norm is equivalent ot the usual norm
of the space $F_{-1}$ and thus the subspace $F_{-1}^-$ can be identified with a closed subspace
of the space  $F_{-1}$. Moreover, we have that
$$
(F_1^-,F_{-1}^-)_{1/2,2}=N.
$$
We now show  that
\begin{equation}\label{eq451}
F_i=F_i^{+}+F_i^-, \ \ P\in L(F_i),\ i=1,-1,
\end{equation}
where the sum is direct.
Let  $i=-1$. Take    $u=u^+ + u^-$, $u^{+}\in M, u^-\in N$. In view of  (\ref{eq421}),
$$
\|u^+\|_{F_{-1}}\leq 2c\sup_{v\in F_1^+} |[u^+,v]|/\|v\|_{F_1}=
$$
$$
=2c\sup_{v\in F_1^+} |[u^+ + u^-,v]|/\|v\|_{F_1}\leq 2c\sup_{v\in F_1} |[u^+ + u^-,v]|/\|v\|_{F_1}\leq 2c\|u\|_{F_{-1}}
$$
Similarly,
$\|u^-\|_{F_{-1}}\leq 2c \|u\|_{F_{-1}}$ and we obtain that
$$
\|u\|_{F_{-1}}\geq \delta (\|u^+\|_{F_{-1}}+\|u^-\|_{F_{-1}}).
$$
These inequalities imply that the projections  $P$  and $(I-P)$ are extensible
to operators of the class  $L(F_{-1})$ and thereby the relation  (\ref{eq451}) is valid for  $i=-1$.
Prove that  $P|_{F_1}\in L(F_1)$. Indeed, for $v\in M$ and $u\in F_1$, we have
$$
|[v,Pu]|=|[v,u]|\leq \|v\|_{F_{-1}}\|u\|_{F_1}.
$$
Therefore, the expression  $[v,Pu]$ admits an extension being a linear
continuous functional over $F_{-1}^+$ and thus there exists an element
 $u^+\in F_1^+$ such that  $[v,Pu]=[v,u^+]$ for all  $v\in M$. As a result,
 $Pu=u^+\in F_1^+$ and the estimate
$$
\|Pu\|_{F_{1}}\leq c_1\sup_{v\in F_{-1}^+} |[v,Pu]|/\|v\|_{F_{-1}}\leq
c_1\sup_{v\in F_{-1}^+} |[v,u]|/\|v\|_{F_{-1}}\leq c_1\|u\|_{F_1}
$$
holds. In this case, we have that  $P,(I-P)\in L(F_1)$ and the equality
 (\ref{eq451}) holds for $i=1$.
Theorem  1.17.1 in  \cite{tr01} implies that
$$
(F_1,F_{-1})_{1/2,2}=(F_1^+,F_{-1}^+)_{1/2,2}+ (F_1^-,F_{-1}^-)_{1/2,2}=M+N=H.
$$
Next, Lemma  \ref{lem1} yields
$(H_1,H_{-1})_{1/2,2}=H$.
\end{proof}

\begin{theorem}\label{th5}
 {Assume that  $L:H\to H$ is an  $m$-$J$-dissipative operator such that  $i{\mathbb R}\subset \rho(L)$,
the conditions  {\rm (\ref{eq12})} and {\rm (\ref{eq18})} are fulfilled, and
there exist maximal nonnegative and maximal nonpositive
subspaces  $M^{\pm}$ of the space  $H$ invariant under  $L$ and such that
$$
\overline{C^{\pm}}\subset \rho(L|_{M^{\pm}}),\ \  H=M^+ + M^-,\
$$
where the sum is direct.
Then the equality  {\rm (\ref{eq14})} holds.
If  $L$ is strictly  $J$-dissipative or uniformly  $J$-dissipative then the corresponding
suspaces  $M^{+}$ and  $M^-$ are definite or uniformly definite, respectively.
}\end{theorem}

\begin{proof} As in Lemma  \ref{lem5}, define the subspaces  $F_1^\pm\subset F_1$ and
$F_{-1}^\pm\subset F_{-1}$. By Lemma  \ref{lem5}, the operators
$\pm L|_{F_{-1}^\pm}: F_{-1}^\pm \to F_{-1}^\pm$ are generators of analytic semigroups.
For example, we consider the operator $L^+=L|_{F_{-1}^+}$. Take $T>0$. As a consequence of the trace
theorem  (\cite[Theorem 1.8.3]{tr01}),
for a  given element $u_0\in (F_1^+,F_{-1}^+)_{1/2,2}$, there exists a function  $v(t)$ such that
$v_t\in L_2(0,T;F_{-1}^+)$, $v\in L_2(0,T;F_{1}^+)$, and  $v(0)=u_0$.
As is well known  (see, for instance, Corollary 1.7 in \cite{kuns}), the Cauchy problem
\begin{equation}\label{eq41}
u_t-L^+u=f, \ \ u(0)=0,
\end{equation}
possesses the maximal regularity property, i.~e., for every
 $f\in L_2(0,T;F_{-1}^+)$, there exists a unique solution to this problem such that $u_t, L^+u\in L_2(0,T;F_{-1}^+)$.
Using the change  $u=v(t)+V(t)$, we reduce the problem
\begin{equation}\label{eq42}
u_t-L^+u=0,\ \ u|_{t=0}=u_0
\end{equation}
to the problem of the form  (\ref{eq41}) and thereby the problem  (\ref{eq42})
has a unique solution such that $u_t, L^+u\in L_2(0,T;F_{-1}^+)$ for all  $u_0\in  (F_1^+,F_{-1}^+)_{1/2,2}$.
Take  $u_0\in F_1^+\subset (F_1^+,F_{-1}^+)_{1/2,2}$.
There exists a constant
$\delta > 0$ such that
$$
S_\delta=\{\lambda : |{\rm arg}\,\lambda| \leq \pi/2 + \delta\}\subset \rho(L^+)
$$
and the resolvent estimates
\begin{equation}\label{eq43}
\|(L^+ - \lambda I)^{-1}f \|_{F_{-1}} \leq c_1 \|f\|_{F_{-1}} (1+ |\lambda|)^{-1}
\end{equation}
hold on every ray  ${{\rm arg}\,\lambda = \theta}\subset  S_{\delta}$. In this case
a solution to the problem  (\ref{eq42}) is representable in the form  (see, for instance,\cite{kuns})
$$
u(t) = \frac{1}{2\pi i} \int_{\Gamma}e^{\lambda t} (L-\lambda I)^{-1} u_0 \,dz,\
$$
where  $\Gamma=\partial S^{\delta}$,
$ S^{\delta}=S_{\delta}\cup \{z:\ |z|<\varepsilon\}$, the integration is perfomed in the positive direction  with respect to
the domain $S^{\delta}$, and the parameter  $\varepsilon$ is so small that  $\{z:\ |z|<\varepsilon\}\subset \rho(L^+)$.
Using the resolvent estimate of Lemma \ref{lem5} and using the arguments of the proof of
Theorem \ref{th3}, we infer
\begin{equation}\label{eq44}
\|u(t)\|_{L_2(0,T;F_{-1})}\leq c\|u_0\|.
\end{equation}
In view of (\ref{eq42}), we conclude that
$
[u_t,u](t)-[Lu,u]=0.
$
Since the expression $[u,u]$ is nonpositive,  integrating this equality in $t$
from $0$ to  $T$, we arrive at the inequality
\begin{equation}\label{eq45}
-2\int_{0}^T \Re [Lu,u](t)\,dt\leq [u_0,u_0].
\end{equation}
The inequality  (\ref{eq18}) and the estimates  (\ref{eq44}) and (\ref{eq45}) imply the inequality
\begin{equation}\label{eq46}
\int_0^T\|u(t)\|^2\,dt\leq c(-\int_{0}^T \Re [Lu,u](t)\,dt+ \|u(t)\|_{L_2(0,T;F_{-1})}^2)\leq c_1\|u_0\|^2.
\end{equation}
Using the estimate  (\ref{eq45}) again, we conclude that
\begin{equation}\label{eq47}
\|u(t)\|_{L_2(0,T;F_1)}^2\leq c_2\|u_0\|^2.
\end{equation}
This estimate and the equation (\ref{eq42}) imply that
\begin{equation}\label{eq48}
\|u_t\|_{L_2(0,T;F_{-1})}^2\leq c_3\|u_0\|^2.
\end{equation}
In this case the relations (\ref{eq47}) and (\ref{eq48}) imply the estimate
\begin{equation}\label{eq49}
\|u_t\|_{L_2(0,T;F_{-1})}^2+ \|u(t)\|_{L_2(0,T;F_1)}^2\leq c_2\|u_0\|^2.
\end{equation}
However, by the trace theorem  (\cite[Theorem 1.8.3]{tr01}), the  left hand-side of this inequality
is estimated from below by  $\delta\|u_0\|_{(F_1^+,F_{-1}^+)_{1/2,2}}^2$, where $\delta $ is a positive constant independent of
 $u_0,u$. This inequality ensures the embedding
 $M^+\subset (F_1^+,F_{-1}^+)_{1/2,2}$ and the estimate
\begin{equation}\label{eq50}
\|u_0\|_{(F_1^+,F_{-1}^+)_{1/2,2}}\leq c_3\|u_0\|.
\end{equation}
Similar arguments prove the embedding
 $M^-\subset (F_1^-F_{-1}^-)_{1/2,2}$ and the estimate
\begin{equation}\label{eq51}
\|v_0\|_{(F_1^-,F_{-1}^-)_{1/2,2}}\leq c_4\|v_0\|,\ \  \forall v_0\in M^-.
\end{equation}
In this case we use an auxiliary problem
\begin{equation}\label{eq52}
u_t+L^-u=0,\ \ u|_{t=0}=v_0\in F_1^-,\ \ L^-=L|_{M^-},
\end{equation}
rather than  the problem (\ref{eq42}).
Let $u_0\in F_1^+$ and let  $v_0\in F_1^-$.
Since the subspaces  $F_1^\pm$ and $F_{-1}^\pm$ are closed subspaces of the spaces
$F_1$ and $F_{-1}$ we have that
$$
K(t,u_0,F_1,F_{-1})\leq K(t,u_0,F_1^+,F_{-1}^+),\
K(t,v_0,F_1,F_{-1})\leq K(t,u_0,F_1^-,F_{-1}^-).
$$
The definition of the norm in the interpolation space yields
\begin{equation}\label{eq53}
\|u_0\|_{(F_1,F_{-1})_{1/2,2}}\leq \|u_0\|_{(F_1^+,F_{-1}^+)_{1/2,2}},\ \
\|v_0\|_{(F_1,F_{-1})_{1/2,2}}\leq \|v_0\|_{(F_1^-,F_{-1}^-)_{1/2,2}}.
\end{equation}
These inequalities, the triangle inequality, and the estimates
(\ref{eq50}) and (\ref{eq51}) guarantee the inequality
$$
\begin{array}{c}
\|u_0+v_0\|_{(F_1,F_{-1})_{1/2,2}}\leq
\|u_0\|_{(F_1^+,F_{-1}^+)_{1/2,2}} + \|v_0\|_{(F_1^-,F_{-1}^-)_{1/2,2}}\leq \vspace{8pt}\\ \leq
c(\|u_0\|+ \|v_0\|)\leq c_1\|u_0+v_0\|_{H},\end{array}
$$
where  $c,c_1$ are positive constants indepedent of  $u_0,v_0$.
The last inequality follows from the fact that the direct sum
of the spaces  $M^+,M^-$ coincides with  $H$.
The  inequality obtained implies that  $H\subset (F_1,F_{-1})_{1/2,2}$.
This containment ensures the equality $H=(F_1,F_{-1})_{1/2,2}$   (see \cite[Chap. 1, Lemma 3.17]{pyat2}) and
thereby  (\ref{eq14}) is true.
Demonstrate that the subspace $M^+$ is positive whenever  $L$ is strictly  $J$-dissipative and uniformly positive
whenever  $L$ is uniformly  $J$-dissipative.
Similar assertions are also valid for the subspace  $M^-$.
Let  $L$ be a strictly $J$-dissipative operator. Assume the contrary that
$M^+$ is degenerate, i.~e., there exists a nonzero element
 $\psi\in M^+$ such that  $[\psi,\psi]=0$.
The Cauchy-Bunyakowski\u{i} inequality for the form  $[\cdot,\cdot]$ on $M^+$
implies that  $\psi\in M^+\cap (M^+)^{[\perp]}$, i. e.,
$[\psi,v]=0$ for all  $v\in M^+$. Since  $0\in \rho(L|_{M^+})$,
we have that
$[L^{-1}\psi,\psi]=0$. On the other hand,
 ${\rm Re\,}[Lu,u]>0$ for all  $u\in D(L)$.
Hence,  ${\rm Re\,}[L^{-1}u,u]>0$ for all  $u\in H$ and thereby
 ${\rm Re\,}[L^{-1}\psi,\psi]>0$; a contradiction.
Let now  $L$ be a uniformly  $J$-dissipative.
Take $u_0\in F_1^+$ and consider a solution $u(t)$ to the problem  (\ref{eq42}).
It satisfies the estimate  (\ref{eq45}). Since $L$ is uniformly $J$-dissipative, the estimate
 (\ref{eq45}) yields
$$
\|u\|_{L_2(0,T;F_1)}^2\leq c[u_0,u_0],
$$
where  $c$ is a constant indepedent of  $u$.
In this case, from the equation  (\ref{eq42}) we have the inequality
$$
\|u_t\|_{L_2(0,T;F_{-1})}^2\leq c_1\|u\|_{L_2(0,T;F_1)}^2\leq c_2[u_0,u_0],
$$
The last two inequalities and the trace thorem (\cite[Theorem  1.8.3]{tr01})
imply that
$$
\|u_0\|_{(F_1,F_{-1})_{1/2,2}}^2\leq  c_3[u_0,u_0].
$$
But we have proven that  $(F_1,F_{-1})_{1/2,2}=H$.
Therefore,  $M^+$ is uniformly positive.
Similarly, we obtain that the subspace  $M^-$ is uniformly negative.
\end{proof}

As a consequence of Theorems  \ref{th2}  and \ref{th5}, we have the following result.

\begin{theorem}\label{th6}
{ Let  $L:H\to H$ be a maximal uniformly  $J$-dissipative operator satisfying
 {\rm (\ref{eq12})}. If the condition {\rm (\ref{eq14})} is fulfilled then
there exist maximal uniformly positive and uniformly negative
subspaces   $M^{\pm}$ of the space  $H$ invariant under  $L$ and such that
$$
\overline{C^{\pm}}\subset \rho(L|_{M^{\pm}}),\ \  H=M^+ + M^-,\
$$
where the sum is direct,
 and the operators  $\pm L|_{M^{\pm}}$ are generators of analytic semigroups.

If there exist maximal uniformly positive and uniformly negative subspaces
 $M^{\pm}$ of the space  $H$ invariant under  $L$ and such that
$$
\overline{C^{\pm}}\subset \rho(L|_{M^{\pm}}),\ \  H=M^+ + M^-,\
$$
where the sum is direct then the equality {\rm (\ref{eq14})} holds and the subspace
 $M^+$ is uniformly positive and  $M^-$ uniformly negative.}
\end{theorem}

Thus, the equality  {\rm (\ref{eq14})} is a necessary and sufficient
condition of existence of the subspaces in question.

\begin{remark}\label{rem1}
All theorems of this section can be strenthened in the sense that the
condition  $i{\mathbb R}\subset \rho(L)$ can be weakened. We can assume that
$i{\mathbb R}\subset \rho(L)$ except for finitely many isolated eigenvalues of finite multiplicity
of the operator $L$. Using the Riesz projections, we can reduce the problem to the already studied case
$i{\mathbb R}\subset \rho(L)$.
\end{remark}

We now present some corollaries in the case when an operator $L:H\to H$ ($H$ is a Krein space with the fundamental symmetry
$J=P^+-P^-$) is representable in the form  (\ref{eq01}). In this case,
the whole space  $H$ with an inner product  $(\cdot,\cdot)$ and the norm  $\|\cdot\|$
is identified with the Cartesian product   $H^+\times H^-$ ($H^{\pm}=R(P^{\pm})$) and the operator
 $L$ with a matrix operator  $L:H^+\times H^-\to H^+\times H^-$ of the form
\begin{equation}\label{eq011}
\begin{array}{c}
L=\left(\begin{array}{cc}A_{11} & A_{12}\\ A_{21} & A_{22}\end{array}\right),\ \vspace{8pt} \\
A_{11}=P^+LP^+,\ A_{12}=P^+LP^-,\
A_{21}=P^-LP^+ ,\ A_{22}=P^-LP^-.\end{array}
\end{equation}
In this case, the fundamental symmetry can be written in the form
$J=\left[\begin{array}{cc}I& 0\\ 0 & -I\end{array}\right]$. Let  $u=(u^+,u^-)\in H$, $v=(v^+,v^-)\in H$.
The inner product  $(u,v)$ in $H$ is written as
$(u^+,v^+)+(u^-,v^-)$ $(u^\pm,v^\pm\in H^\pm)$ and an idefinite inner product as  $(Ju,v)=(u^+,v^+)-(u^-,v^-)$.
Describe the conditions of Theorem \ref{th1} in this case.
Put $L_0=\left(\begin{array}{cc}A_{11} & 0\\ 0 & A_{22}\end{array}\right)$.

\begin{theorem}\label{th7}
{ Assume that $L:H\to H$ is an  $m$-$J$-dissipative operator
in a Krein space  $H=H^+ \times H^-$
such that   $i{\mathbb R}\in \rho(L)\cap \rho(L_0)$, the condition  {\rm (\ref{eq200})} holds,
$D(L)=D(L_0)$,  and there exists numbers $\lambda\in \rho(L)$, $\mu\in \rho(L_0)$ such that
the operator  $(L-\lambda I)^{-1}(L_0-\mu I): D(L_0)\to D(L)$ is extensible to an isomorphism
of $H$ onto $H$. Then there exist maximal nonnegative and maximal nonpositive
subspaces  $M^\pm$ invariant under $L$. The whole space  $H$ is representable as the direct sum
  $H  = M^+  + M^-$,
$\sigma(\mp L|_{M^\pm})\subset {\mathbb C}^\pm$, and the operators  $\pm L|_{M^\pm}$ are generators of analytic semigroups.}
\end{theorem}

\begin{proof}
Note that the condition of the theorem implies that the operators
 $A_{11}:H^+\to H^+$ and $-A_{22}:H^-\to H^-$ are $m$-dissipative.
To refer to Theorem  \ref{th1}, it suffices to prove that
the interpolation equality  (\ref{eq14}) holds. By condition,
$D(L)=D(L_0)=H_1=H_1^+\times H_1^-$, where  $H_1^+=D(A_{11})$ and $H_1^-=D(A_{22})$.
The last condition of the theorem implies the existence of constants
 $c_1,c_2>0$ such that
$$
c_1\|(L-\lambda I)^{-1}u\|\leq \|(L_0-\mu I)^{-1}u\|\leq c_2 \|(L-\lambda I)^{-1}u\|\ \ \forall u\in H
$$
and thus the space  $H_{-1}$ which is a completion of $H$ with respect to the norm $\|L-\lambda I)^{-1}u\|$ coincides with
the same space for the operator  $L_0$, i.~e., with the space  $H_{-1}^+\times H_{-1}^-$, where
 $H_{-1}^\pm$ are completions of  $H^\pm$ with respect to the norms
 $\|(A_{11}-\mu I)^{-1}u^+\|$ and $\|(A_{22}-\mu I)^{-1}u^-\|$.
Proposition  2 implies that  $(H_1^+,H_{-1}^+)_{1/2,2}=H^+$ and
$(H_1^-,H_{-1}^-)_{1/2,2}=H^-$. Applying Theorem  1.17.1 of \cite{tr01} yields
 $(H_1,H_{-1})_{1/2,2}=(H_1^+,H_{-1}^+)_{1/2,2}\times (H_1^-,H_{-1}^-)_{1/2,2}$ and thereby
$(H_1,H_{-1})_{1/2,2}=H^+ \times H^-=H$.
\end{proof}

In the next theorem we assume that
 $A_{11},-A_{22}$ are $m$-dissipative operators satisfying the conditions
\begin{equation}\label{la1}\begin{array}{c}
\exists c>0:\ |(A_{11}u^+,v^+)|\leq c\|u^+\|_{F_1^+}\|v^+\|_{F^+_1},  \\
|(A_{22}u^-,v^-)|\leq c\|u^-\|_{F_1^-}\|v^-\|_{F^-_1}\end{array}
\end{equation}
for all  $u^\pm,v^\pm\in H_1^\pm$,
where $H_1^+=D(A_{11})$, $H_1^-=D(A_{22})$, and
$$
\|u^+\|_{F_1^+}^2=\Re (-A_{11}u^+,u^+)+ \|u^+\|^2,\ \
\|u^-\|_{F_1^-}^2=\Re (A_{22}u^-,u^-)+ \|u^-\|^2.
$$
We denote by $F_1^\pm$ the completions of
 $D(A_{11})$ and $D(A_{22})$ with respect to the norms  $\|\cdot\|_{F_1^+}$ and $\|\cdot\|_{F_1^-}$,
respectively. By $F_{-1}^\pm$, we mean the negative spaces constructed on the pairs
$F_1^\pm,H^\pm$.
We also assume that the operators  $A_{12}$ and $A_{21}$ are subordinate to the operators
 $A_{11}$ and $A_{22}$ in the following sense:
$D(A_{22})\subset D(A_{12})$, $D(A_{11})\subset D(A_{21})$, and
\begin{equation}\label{la2}
\exists c>0:\ \|A_{12}u^-\|_{F_{-1}^+}\leq c\|u^-\|_{F_1^-},\  \
\|A_{21}u^+\|_{F_{-1}^-}\leq c\|u^+\|_{F_1^+}\ \forall u^\pm\in H_1^\pm;
\end{equation}
\begin{equation}\label{la3}
\exists c_0>0:\  \|u^+\|_{F_1^+}^2+ \|u^-\|_{F_1^-}^2\leq c_0
(\Re(-L\vec{u},\vec{u})+\|\vec{u}\|^2)
\end{equation}
for all  $ \vec{u}=(u^+,u^-)\in H_1=H_1^+\times H_1^-.$

\begin{theorem}\label{th8}
{ Let $L:H\to H$ be an $m$-$J$-dissipative operator such that
 $i{\mathbb R} \subset \rho(L)$, and the conditions  {\rm (\ref{la1})-(\ref{la3})} hold.
Then there exist maximal nonnegative and maximal nonpositive
subspaces  $M^\pm$ invariant under $L$. The whole space  $H$ is representable as the direct sum
  $H  = M^+  + M^-$,
$\sigma(\mp L|_{M^\pm})\subset {\mathbb C}^\pm$, and the operators  $\pm L|_{M^\pm}$ are generators of analytic semigroups.
If in addition $L$ is strictly $J$-dissipative or uniformly $J$-diisipative then the corresponding spaces $M^+$ and $M^-$
can be chosen to be positive and negative or uniformly positive and uniformly negative, respectively.}
\end{theorem}

\begin{proof}
The condition  (\ref{la3}) implies that the norm in the space $F_1$ is equivalent to  the norm
$\|u\|_{F_1}^2=\|u^+\|_{F_1^+}^2+\|u^-\|_{F_1^-}^2$ $(u=(u^+,u^-))$.
The definition of the space $F_{-1}$ and the indefinite inner product in
$H$ ensures that  $F_{-1}=F_{-1}^+\times F_{-1}^-$. Proposition  \ref{pro3} and Theorem
1.17.1 in \cite{tr01} yields
$(F_1,F_{-1})_{1/2,2}=(F_1^+,F_{-1}^+)_{1/2,2}\times (F_1^-,F_{-1}^-)_{1/2,2}
=H^+ \times H^-=H$.  The claim results from Theorem \ref{th2}.
\end{proof}

\end{document}